\documentclass[10pt,amsfonts, epsfig]{amsart}
\usepackage{amsmath, amscd, amssymb}
\usepackage{graphpap, color}
\usepackage[mathscr]{eucal}
\usepackage{pstricks}
\usepackage{color}
\usepackage{cancel}
\usepackage[mathscr]{eucal}
\usepackage{pstricks}
\usepackage{color}
\usepackage{cancel}
\usepackage{verbatim}

\def\fC{{\mathbf C}}

\numberwithin{equation}{section}

\newcommand{\Spec}{\operatorname{Spec}}

\def\sO{{\mathscr O}}

\def\sO{\mathscr{O}}

\def\sF{\mathscr{F}}
\def\sV{\mathscr{V}}

\newcommand{\LL}{\mathbb{L}}

\newcommand{\bk}{\mathbf{k}}

\newcommand{\kk}{\bk}

\newcommand{\cal}{\mathcal}

\def\cC{{\cal C}}
\def\cD{{\cal D}}
\def\cE{{\cal E}}

\def\cM{{\cal M}}

\def\cO{{\cal O}}

\def\fC{\mathfrak{C}}
\def\fD{\mathfrak{D}}

\def\mapright#1{\,\smash{\mathop{\lra}\limits^{#1}}\,}

\def\dual{^{\vee}}

\def\sta{^\ast}

\def\virt{^{\mathrm{vir}}}
\def\upmo{^{-1}}
\def\sta{^{\ast}}

\def\mm{{\mathfrak m}}
\def\sta{^*}

\def\lbe{_{\beta}}
\def\lra{\longrightarrow}

\def\lsta{_{\ast}}

\def\begeq{\begin{equation}}
\def\endeq{\end{equation}}
\def\and{\quad{\rm and}\quad}
\def\bl{\bigl(}
\def\br{\bigr)}
\def\defeq{:=}

\def\sub{\subset}

\def\Po{{\mathbb P^1}}
\def\and{\quad\text{and}\quad}

\def\lalp{_\alpha}

\def\ob{\text{ob}}

 \DeclareMathOperator{\Ext}{Ext}
  \DeclareMathOperator{\Hom}{Hom}

\DeclareMathOperator{\id}{id}

 \DeclareMathOperator{\rank}{rank}

\newtheorem{prop}{Proposition}[section]
\newtheorem{theo}[prop]{Theorem}
\newtheorem{lemm}[prop]{Lemma}
\newtheorem{coro}[prop]{Corollary}

\newtheorem{defi}[prop]{Definition}

\newtheorem{defi-theo}[prop]{Definition-Theorem}
\newtheorem{defi-prop}[prop]{Definition-Proposition}

\DeclareMathOperator{\coker}{coker}

\def\Ob{\cO b}

\def\lalpbe{_{\alpha\beta}}
\def\bul{^\bullet}

\def\sta{^\ast}

\let\label=\label

\def\sO{{\mathscr O}}

\def\beq{\begin{equation}}
\def\eeq{\end{equation}}

\def\vsp{\vskip5pt}

\def\fc{\mathfrak c}

\def\bR{\mathbf R}
\def\psta{^{\prime\ast}}

\def\bee{\begin{equation}}
\def\eeq{\end{equation}}

\def\bH{{\mathbf H}}

\def\ti{\tilde}

\def\Lam{{\Lambda}}

\def\lred{_{\mathrm{red}}}

\def\pdual{^{\prime\vee}}

\let\label=\label

\title[Semi perfect obstruction theory and DT invariants]
{Semi-perfect obstruction theory and DT invariants of derived objects}

\date{}

\author{Huai-liang Chang and }
\address{Department of Mathematics, Hong Kong University of Science and Technology}
\email{mahlchang@ust.hk}

\author{Jun Li}
\address{Department of Mathematics, Stanford University} \email{jli@math.stanford.edu}

\begin{document}
\maketitle
 
\begin{abstract}
We introduce a semi-perfect obstruction theory of a Deligne-Mumford stack $X$ that consists of local perfect obstruction 
theories with a global obstruction sheaf.
We construct the virtual cycle of a Deligne-Mumford stack with a semi-perfect obstruction theory.
We use semi-perfect obstruction theory to construct virtual cycles of
moduli of derived objects on Calabi-Yau threefolds. 
\end{abstract}

\section{Introduction}
\def\fOb{\mathfrak{Ob}}
In this note, we introduce the notion of semi-perfect obstruction of a Deligne-Mumford stack. This notion 
has the advantage of the two perfect obstruction theories introduced in \cite{LT} and \cite{BF}, by combining
the local version of the perfect obstruction theory 
formulated in \cite{BF} with the locality of the virtual normal cone proved in \cite{LT}. This construction
applies to moduli spaces that do not have universal families, like moduli of derived objects;
it makes working with virtual cycles flexible.

Let $X\to\cM$ be a representable morphism from a Deligne-Mumford stack to a smooth Artin stack of pure dimension.
A semi-perfect relative obstruction theory $\phi$ of $X\to \cM$ consists of an \'etale cover $U\lalp\to X$ and (truncated) perfect obstruction
theory $\phi\lalp: E\lalp\to \LL_{U\lalp/\cM}^{\bullet\ge -1}$ 
such that the obstruction sheaves
$\Ob_{\phi\lalp}= H^1(E\lalp\dual)$ descend to an obstruction sheaf on $X$, and that the infinitesimal 
obstruction assignments of deforming closed points in $X$ are independent of the charts $U\lalp$.

We denote by $\Ob_\phi$ the sheaf stack of the obstruction sheaf $\Ob_\phi$ of the semi-obstruction\ theory $\phi$.
We let $s$ be the zero section of $\Ob_\phi$.
We will make sense of the group of cycles $Z\lsta \Ob_\phi$, and construct a Gysin map $s^!: Z\lsta \Ob_\phi\to A\lsta X$.

%

\begin{theo}
Suppose $X\to\cM$ as stated has a 
semi-perfect relative obstruction theory $\phi$. 
Then the collection of intrisic normal cones 
$\cC_{U\lalp/\cM}\sub h^1/h^0((\LL_{U\lalp/\cM}^\bullet)\dual) )$
pushforward and glue to a cycle 
$[\fc_{X/\cM}]\in Z\lsta \Ob_\phi$.
We define the virtual class of $(X,\phi)$ be
$$[X,\phi]\virt=s^![\fc_{X/\cM}]\in A\lsta X.
$$
This class has the usual properties satisfied by virtual classes.
\end{theo}

Using the construction of the moduli spaces of derived objects by
Inaba \cite{Inaba} and Lieblich \cite{Lieb}, and the perfect obstruction of derived objects over Calabi-Yau threefolds by
Huybrecht-Thomas \cite{HT}, we show that the moduli of derived objects has a semi-perfect obstruction theory. Applying the
main theorem, we obtain its virtual cycle, whose degree is the Donaldson-Thomas invariant of the moduli space,
and its deformation invariance property.


\vsp
The new input of this paper is that we put the 
cycle associated to the coarse 
moduli of the intrinsic normal cone of a Deligne-Mumford stack into its
obstruction sheaf, assuming the existence of a semi-perfect obstruction theory.
Historically, the virtual cycle construction of Tian and the second named author \cite{LT}  constructed the virtual normal
cone in a vector bundle that surjects onto the obstruction sheaf; that of Behrend-Fantechi \cite{BF} constructed the intrinsic
normal cone in the normal sheaf of the stack and then put it in the bundle-stack $h^1/h^0$ of the obstruction theory. After 
Kresch's \cite{Kre}
construction, one can apply the Gysin map to a cycle in the bundle-stack 
directly, bypassing
the need of the mentioned global vector bundles. Later,
in \cite{Li2} and  \cite{Li3} the intrinsic normal cone was implicitly put inside the obstruction sheaf. For moduli spaces that lack universal
families but do have obstruction sheaves (of their semi-perfect obstruction theories), putting their intrinsic normal cones in the obstruction sheaves
allows one to apply the Gysin map to obtain their virtual cycles.

\vsp
\noindent
{\bf Convention}. In this paper, for a morphism $\rho: V\to U$ that is either a closed embedding, an
\'etale morphism, or a composition of both, we will use restricting to $V$, i.e. $\cdot |_V$,
to denote the pullback under $\rho$. For instance, given a derive
object $E$ on $U$, we denote by $E|_V=\rho\sta E$; for a stack $Y$ over $U$, we
denote by $Y|_V=Y\times_U V$, and for $\phi: E\to F$ a homomorphism of derived objects on $U$,
we denote by $\phi|_V$ the pullback homomorphism $\rho\sta \phi: \rho\sta E\to \rho\sta F$.
 
We fix a characteristic zero algebraically closed field $\kk$. All schemes and stacks in this paper are 
defined over $\kk$.

\vsp
\noindent{\bf Acknowledgment}: We thank T. Nevins  and S. Katz for pointing out an oversight of the original version of 
this paper. The first named author is supported by a Hong Kong DAG grant. 
The second named author is partially supported by
an NSF grant and a DARPA grant. 

\section{Intrinsic normal cones in the obstruction sheaves} 
\def\Coh{\, \mathrm{Coh}}

In this section, we put the intrinsic normal cone in the obstruction sheaf, and show that it only depends on the
the obstruction class assignments.
%

We let $\cM$ be as before, which is a pure dimensional smooth Artin stack of finite presentations.
Let $U\to \cM$ be a morphism from a scheme of finite type
to $\cM$.

\begin{defi}[\cite{BF}]\label{truncated}
A (truncated) perfect (relative) obstruction theory of $U\to \cM$ consists of a morphism in $D(U)$ 
\beq\label{ob-1}
\phi_{}:E_{} \lra L_{U/\cM}
\eeq
such that
\begin{enumerate}
\item $E_{}$ is a perfect complex in $D(U)$ of amplitude contained in $[-1,0]$;
\item $h^0(\phi_{})$ is an isomorphism, and $h^{-1}(\phi_{})$ is surjective.
\end{enumerate}
Given $\phi$, we call $\Ob_\phi\defeq h^1(E\dual)$ its obstruction sheaf.
\end{defi}
  
Following \cite{BF} and \cite{Kresch,Kre}, we denote the intrinsic normal cone and the intrinsic normal sheaf of $U/\cM$ by
\beq\label{cone-inclusion}
\cC_{U/\cM}\sub N_{U/\cM}\defeq  h^1/h^0((L_{U/\cM})\dual)\cong h^1/h^0((\LL_{U/\cM})\dual).
\eeq
We denote its associated cycle by 
$[\cC_{U/\cM}] \in Z\lsta N_{U/\cM}$. 
Let
\beq\label{embed}
h^1/h^0(\phi_{}\dual): N_{U/\cM}\mapright{\subset} h^1/h^0(E_{}\dual)
\eeq
be induced by the truncated perfect obstruction theory.

In this paper, for a coherent sheaf $\sF$ of $\sO_U$-modules. We form the sheaf stack of $\sF$, which 
is the groupoid that associates to any $\rho: S\to U$ the set $\Gamma(S, \rho\sta \sF)$. By abuse of notation,
we denote the sheaf stack of $\sF$ by the same symbol $\sF$.
And following the $h^1/h^0$ notation, we denote by $h^1(E\dual)$ the sheaf stack of the cohomology sheaf
$H^1(E\dual)$.
Applying this to $\Ob_\phi=H^1(E\dual)$ and composed with \eqref{embed}, 
we have the induced composite morphism of stacks
\beq\label{eta}
\eta_\phi: N_{U/\cM}\lra h^1/h^0(E\dual)\lra h^1(E\dual).
\eeq

\begin{defi}\label{integral}
We call a substack $A\sub \sF$
a reduced cycle if for any locally free sheaf $\sV$ of $\sO_U$-modules and a surjective $f: \sV\to \sF$, and denoting by the same
$f$ the induced morphism of their respective stacks $\sV\to \sF$, 
$\sV\times_{\sF} A$ is a reduced Zariski closed subset of $\sV$. Given three reduced cycles $A_1$, $A_2$ and $A_3$ of
$\sF$, we say $A_3=A_1\cup A_2$ if for the $f: \sV\to \sF$ as before,
 $\sV\times_{\sF} A_3=\sV\times_{\sF} A_1\cup \sV\times_{\sF} A_2$ as subschemes of $\sV$.
We say $A$ is integral if it is reduced and is
not a union of two distinct non-empty reduced cycles of $\sF$.
We define $Z\lsta \sF$ to be the (rational) linear combinations of integral cycles in $\sF$.
\end{defi}

\begin{lemm}
Let $A\sub h^1/h^0(E\dual)$ be an integral cycle. Then the image $\eta_\phi(A)\sub h^1(E\dual)$ is an integral
cycle in $h^1(E\dual)$. We call the cycle $[\eta_\phi(A)]$ the push-forward of $[A]$, and denote it by
$\eta_{\phi\ast}[A]\in Z\lsta h^1(E\dual)$.
\end{lemm}

\begin{proof}
We let $E=[E_{-1}\to E_0]$ with both $E_i$ locally free; let $\sV=E_{-1}\dual$ and let $\sV[-1]\to E\dual$ be induced
by the identity $\sV\to E_{-1}\dual$. By viewing $\sV$ as the sheaf stack of $\sV$, the arrow $\sV[-1]\to E\dual$
induces a morphism $p: \sV\to h^1/h^0(E\dual)$. 

Let $A\sub h^1/h^0(E\dual)$ be an integral Artin substact.
By definition, 
$$p\upmo(A)\defeq A\times_{h^1/h^0(E\dual)} \sV\sub \sV
$$
is a reduced Zariski closed subset.
Let $p': \sV\to h^1(E\dual)$ be the morphism of stack induced by the same $\sV[-1]\to E\dual$. It fits into 
a commutative square
$$\begin{CD} \sV @= \sV\\
@VV{p}V @VV{p'}V\\
h^1/h^0(E\dual) @>{q}>> h^1(E\dual).
\end{CD}
$$
We claim that 
$$p^{\prime-1}(q(A))\defeq q(A)\times_{h^1(E\dual)} \sV =\ p\upmo(A)\sub \sV.
$$

We prove that $p^{\prime-1}(q(A))\sub p\upmo(A)$; the other direction of inclusion is similar and more direct. 
Let $T$ be any affine scheme and $\rho: T\to U$ a morphism. By definition, an object in $p^{\prime-1}(q(A))(\rho)$
consists of a pair of a $T$-morphism $\xi_1: T\to \sV|_T$ and an $E_0\dual|_T$-equivariant $T$-morphism $\xi_2: P\to E_{-1}\dual|_T$,
where $P$ is a principle $E_0\dual|_T$-bundle over $T$, such that 
\beq\label{eq-prin}
p'\circ\xi_1=q\circ\xi_2/E_0\dual,
\eeq
where $q\circ\xi_2/E_0\dual$ is the descent of $q\circ \xi_2$ to $T\to h^1(E\dual)|_T$ using that $E_0\dual$ acts on
$h^1(E\dual)$ trivially. 

We need to show that $\xi_1$ is an object in $p\upmo(A)(\rho)$. For this, it suffices to show that $p(\xi_1)\cong \xi_2$. 
Following \cite[Sect. 2]{BF}, $p(\xi_1)$ is given by the $E_0\dual|_T$-principle bundle $P'=E_0\dual|_T$ with the 
tautological projection
$\pi': P'\to T$ and a
$\xi_2': P'\to E_{-1}\dual|_T$ via $\xi_2'(v)=d\dual(v)+\xi_1(\pi'(v))$, where $d\dual: E_0\dual\to E_{-1}\dual$ is the dual of
$[E_{-1}\to E_0]$. 

Because $T$ is affine, we can find a section $s: T\to P$ of the bundle $P\to T$. Then \eqref{eq-prin} implies
that 
$$p'\circ\xi_1=(q\circ p)\circ \xi_2\circ s=p'\circ\xi_2\circ s: T\lra  h^1(E\dual)|_T.
$$
Therefore, there is a section $t: T\to E_0\dual|_T$ so that $\xi_2\circ s=\xi_1+d\dual\circ t: T\to \sV$.
Therefore, we can find an isomorphism $P\cong P'$ as principle bundle so that
$\xi_2'\cong \xi_2$. This proves that $p(\xi_1)\cong \xi_2$; namely, $\xi_1$ is an object in $p\upmo(A)$.
This proves the Lemma.
\end{proof}

Representing $[\cC_{U/\cM}]$ as a linear combination of integral cycles in $N_{U/\cM}$, and 
applying the push-forward $\eta_{\phi\ast}$, we obtain 
$$[\mathfrak c_{\phi}]=\eta_{\phi\ast}[\cC_{U/\cM}]\in Z\lsta  h^1(E\dual).
$$

In the remainder of this subsection, we study the dependence of this cycle on the obstruction theory $\phi$.
Suppose we have another truncated perfect relative obstruction theory
\beq\label{ob-2}
\phi_{}': E_{}'\lra L_{U/\cM},
\eeq
and suppose we have an isomorphism 
\beq\label{iso-m}
\psi: h^1(E\dual)\mapright{\cong} h^1(E^{\prime\vee}).
\eeq
We study when the cycles $\psi\lsta[\fc_\phi]=[\fc_{\phi'}]$.

As was proved in \cite{LT}, the cycle $[\fc_\phi]$ \black is determined by the
obstruction theories to deforming closed points in $X$. The uniqueness proof given here
reminiscent to that in \cite{LT}.

For any closed $p\in U$, we denote
\beq\label{TT}
T^i_{p,U/\cM}=H^i((L_{U/\cM})\dual|_p).
\eeq

\begin{defi}\label{def1.7}
We define the intrinsic obstruction space to deforming $p\in U$ be
$T^1_{p,U/\cM}$;
we define the obstruction space (of the obstruction theory $\phi_{}$) to deforming  $p\in U$ be
$\mathrm{Ob}(\phi,p)=H^1(E_{}\dual|_p)$.
\end{defi}

\begin{defi} \label{def-22}Let $\iota: T\to T'$ be a closed subscheme with $T'$ local Artinian. Let $I$ be the ideal sheaf of $T$ in $T'$,
and let $\mm$ be the ideal sheaf of the closed point of $T'$.
We call $\iota$ a small extension if $I\cdot \mm=0$.
Given a small extension $(T, T',I,\mm)$ that fits into a commutative square
\beq\label{lift}
\begin{CD}
  T@>{g}>>U\\
  @VV{\iota}V@VVV\\
T'@>>> \cM
\end{CD}
\eeq
so that the image of $g$ contains a closed point $p\in U$,
finding a morphism $g': T'\to U$ that commutes with the arrows in \eqref{lift} is called
``infinitesimal lifting problem of $U/\cM$ at $p$''.
\end{defi}

Applying the standard obstruction theory and using
\beq\label{formula}
\Ext^i(g^\ast L_{U/\cM},I) = \Ext^i(g^\ast\LL^{\bullet}_{U/\cM},I), \quad i=0,1,
\eeq
we obtain

\begin{lemm}[{\cite[Chap. 3, Thm. 2.1.7]{Illusie}}]
For an infinitesimal lifting problem of $U/\cM$ at $p$ as in \eqref{lift}, 
there is a canonical element
$$\omega(g,T,T')\in \Ext^1(g^\ast L_{U/\cM},I)=T^1_{p, U/\cM}\otimes_{\kk}I
$$
whose vanishing is necessary and sufficient for the lifting problem to be solvable;
in case the lifting problem is solvable, the collection of the solutions form a torsor under
$$\Ext^0(g^\ast L_{U/\cM},I)=\Hom(g^\ast\Omega_{U/\cM},I).$$
\end{lemm}

\begin{defi}
Let $\phi_{}$ in \eqref{ob-1} be a perfect obstruction theory. For the infinitesimal lifting problem
\eqref{lift}, we call the image
\beq\label{ob-class}
\ob(\phi_{}, g,T,T')\defeq H^1(\phi_{}\dual)(\omega(g, T, T'))\in \Ext^1(g\sta E,I)=\mathrm{Ob}(\phi,p)\otimes_{\kk} I
\eeq
the obstruction class (of $\phi$) to the lifting problem \eqref{lift}.
\end{defi}

\begin{coro}[\cite{BF}]
Let $\phi$ in \eqref{ob-1} be a perfect relative obstruction theory. Then
the lifting problem \eqref{lift} is solvable if and only if $\ob(\phi,g,T,T')=0$.
\end{coro}

\begin{proof}
This is true because $H^{1}(\phi_{}\dual|_p)$ is injective.
\end{proof}

We now back to the pair of obstruction theories $\phi$ and $\phi'$ mentioned in \eqref{ob-2}. 

%
%
%
%

\begin{defi}\label{def1.8}
We call $\phi_{}$ and $\phi'$ $\nu$-equivalent
if there is an isomorphism of sheaves
\beq\label{psi}
\psi: H^1(E\dual) \mapright{\cong} H^1(E^{\prime\vee})
\eeq
so that for every closed point $p\in U$, and for any ``infinitesimal lifting problem of $U/\cM$ at $p$" as in (\ref{lift}), 
we have
$$\psi|_p\bl \ob(\phi,g,T,T')\br=\ob(\phi',g,T,T')\in \mathrm{Ob}(\phi',p)\otimes_{\kk} I.
$$ 
\end{defi}

%

%
%
The main result of this section is

\begin{prop}\label{glue-cycle} 
Let \eqref{psi} be a $\nu$-equivalence of
$\phi_{}$ and $\phi'$, and let 
$\eta_\phi: N_{U/\cM} \to h^1(E\dual)$ be the
induced morphisms of stacks (cf. \eqref{eta}). Then for any integral cycle $A\sub N_{U/\cM}$,
$$\psi\lsta\bl \eta_{\phi\ast} [A]\br=\eta_{\phi'\ast}[A]\in Z\lsta h^1(E^{\prime\vee}).
$$
\end{prop}

We recall the following known facts. 

\begin{lemm}\label{coro-1}
Suppose $G$ is a perfect complex over $U$ of amplitude $\le 0$.
Let $p\in U$ be a closed point
and denote $V^i=H^i(G\dual|_p)$. 
Then canonically,
$h^1/h^0(G\dual)|_p\cong  [V^1/V^0]$, where $V^0$ acts on $V^1$ trivially.
\end{lemm}

\begin{proof}
Let $G=[\cdots\mapright{d_{-1}} G_{-1}\mapright{d_0} G_0\to 0]$, where $G_i$ are locally free.
Without loss of generality, we assume $d_0|_p: G_{-1}|_p\to G_0|_p$ is trivial.
Let $G'_{-1}=\coker(d_{-1})$. Then $G^{\geq -1}=[G'_{-1}\to G_0]$. 
%
We let $C=\Spec\text{Sym}(G'_{-1})$. Then the arrow $G'_{-1}\to G_0$ defines an action
$G_0\dual\times_Y C\to C$, where we view $G_0\dual$ as the total space of the bundle $G_0\dual$. 
By definition (cf. \cite{BF}), $h^1/h^0(G\dual)=[C/G_0\dual]$, where $[\cdot]$ means the quotient stack.

By the base change property of $\coker(d_{-1})$, we have $\coker(d_{-1}|_p)=G_{-1}'|_p$. 
Since $d_0|_p=0$,
we have  $C\times_U p=V^1$. 
The lemma then follows from the definition of $h^1/h^0$ construction.
\end{proof}

%
%
%
%

We quote the following

\begin{lemm}[{\cite[Prop 4.7]{BF}}] \label{any}
Let the situation be as in Definition \ref{def-22}, and let $p$ be a closed point of $U$.
Given any vector $v\in T^1_{p,U/\cM}$, there exists an infinitesimal lifting problem of $U/\cM$ at $p$
as in Definition \ref{def-22} 
such that for some $w\ne 0\in I$, the obstruction class 
$$w(g,T,T')=v\otimes w \in \Ext^1(g^\ast L_{U/\cM},I)\cong
T^1_{p,U/\cM}\otimes_{\kk(p)} I.
$$
\end{lemm}
    
%
%
%

\begin{proof}[Proof of Proposition \ref{glue-cycle}]
Let $E'=[E_{-1}'\to E_0']$ with both $E_i'$ locally free. Let $\sV=E_{-1}\pdual$.
To prove the Proposition, by definition, we need to show that for the surjective morphism $\sV\to h^1(E\pdual)$,
\beq\label{sub-A}
\psi(\eta_{\phi}(A))\times_{h^1(E\pdual)} \sV=
\eta_{\phi'}(A)\times_{h^1(E\pdual)}\sV
\eeq
as subsets in $\sV$. Since $\sV$ is a scheme and $A$ is integral, both sides of \eqref{sub-A} are reduced and
Zariski closed. Thus to show \eqref{sub-A}, it suffices to check that for any
closed $p\in U$, we have
$$
\psi(\eta_{\phi}(A))\times_{h^1(E\pdual)} (\sV|_p) =
\eta_{\phi'}(A)\times_{h^1(E\pdual)}(\sV|_p).
$$
But this follows from that there is an isomorphism $\tau_p$ making the following square commutative:
\beq\label{AA}
\begin{CD}
 N_{U/\cM}|_p\,\defeq\,h^1/h^0(L\dual_{U/\cM})|_p @>{\eta_\phi|_p}>> h^1(E\dual)|_p\\
 @VV{\tau_p}V  @VV{\psi|_p}V\\
 N_{U/\cM}|_p\,\defeq\,h^1/h^0(L\dual_{U/\cM})|_p @>{\eta_{\phi'}|_p}>> h^1(E^{\prime\vee})|_p.\\
\end{CD}
\eeq
Applying Lemma to $G=L_{U/M}$, we obtain canonical isomorphisms
$$h^1/h^0((L_{U/\cM})\dual)|_p\cong [H^1(L_{U/\cM}|_p)\dual/H^0(L_{U/\cM}|_p)\dual]=[T^1_{p,U/\cM}/T^0_{p,U/\cM}];
$$
applying the same Lemma to $G=E$, we obtain the identity
$h^1(E\dual)|_p = \text{Ob}(\phi,p)$.
Similarly we obtain isomorphisms with $\phi$ and $E$ replaced by $\phi'$ and $E'$.
Therefore, the existence of $\tau_p$ making \eqref{AA} commutative if the
square
$$\begin{CD}
T^1_{p,U/\cM}@>{H^1(\phi_{}\dual|_p)}>> \text{Ob}(\phi,p)\\
@| @VV{\psi|_p}V\\
T^1_{p,U/\cM}@>{H^1(\phi^{\prime\vee}|_p)}>> \text{Ob}(\phi',p)\\
\end{CD}
$$ 
is commutative.

We prove this commutativity. Given any $v\in T^1_{p,U/\cM}$, by Lemma \ref{any} there exists an infinitesimal
lifting problem $(g,T, T')$ for $U/\cM$ at $p$
as in \eqref{lift} and a $0\neq w\in I$ such that 
$$v\otimes w=\omega(g,T,T')\in T^1_{p,U/\cM}\otimes_{\kk(p)} I.
$$
Since $\phi_{}$ and $\phi'_{}$ are numerically equivalent, we have
$$H^1(\phi_{}\dual|_p)(v)\otimes w=H^1(\phi_{}\dual|_p)(v\otimes w)=\qquad\qquad\qquad\qquad
\qquad\qquad
$$
$$\qquad\qquad=
H^1(\psi\dual|_p)\bl H^1(\phi^{\prime\vee}|_p)(v\otimes w)\br=H^1(\psi\dual|_p)\bl H^1(\phi^{\prime\vee}|_p)(v)\br\otimes w.$$
As $w\neq 0$ and $I$ is a $\kk(p)$ vector space, we have 
$$H^1(\phi_{}\dual|_p)(v)=H^1(\psi\dual|_p)\bl H^1(\phi^{\prime\vee}|_p)(v)\br.
$$
This proves the Proposition.
\end{proof}

\section{Semi perfect obstruction theory}
\def\abg{{\alpha\beta\gamma}}

Let $\cM$ be as before and let $X$ be a Deligne-Mumford stack of locally finite type with a morphism $X\to \cM$.
Given two schemes $U\lalp$ and $U\lbe$ with \'etale $U\lalp, U\lbe\to X$, we 
denote $U\lalpbe=U\lalp\times_X U\lbe$, and for any derived object $F\in D(U\lalp)$ we denote by
$F|_{U\lalpbe}$ the pull back of $F$ under the projection $U\lalpbe\to U\lalp$.

\begin{defi}\label{g-semi}
A semi-perfect relative obstruction theory of $X\to \cM$ consists of
an \'etale covering $\{U\lalp\}_{\alpha\in\Lam}$ of $X$ by affine schemes, and
truncated perfect relative obstruction theories
$$\phi_{\alpha}: E\lalp \lra L_{U\lalp/\cM},\ \alpha\in\Lam
$$
such that   \\
{1}. for each pair $\alpha, \beta\in\Lam$ there is an isomorphism
\beq\label{descent}
\psi\lalpbe:H^1(E\lalp\dual)|_{U\lalpbe}\mapright{} H^1(E\lbe\dual)|_{U\lalpbe}
\eeq
so that the collection $(H^1(E\lalp\dual), \psi\lalpbe)$ forms a descent data of sheaves.\\
{2}.  for any pair $\alpha,\beta\in\Lam$, the obstruction theories $\phi\lalp|_{U\lalpbe}$ and
$\phi\lbe|_{U\lalpbe}$ are $\nu$-equivalent via $\psi\lalpbe$.
\end{defi}

Obviously, a perfect obstruction theory 
is a semi-perfect obstruction theory.
\vsp

We fix a
semi-perfect obstruction theory $\phi=\{\phi\lalp,U\lalp, E\lalp,\psi\lalpbe\}_\Lam$.
We denote by $\Ob_{\phi}$ the resulting descent sheaf on $X$ from (1) of Definition \ref{g-semi};
we call it the obstruction sheaf of the semi-perfect obstruction theory.

Let $\sF$ be a coherent sheaf of $\sO_X$-modules, viewed as a sheaf stack.

\begin{defi}
A reduced cycle $A$ of $\sF$ is a substack $A\sub\sF$ so that for
any \'etale open $U\to X$, $U\times_{\sF}A\sub \sF|_U$ is a reduced cycle in the sense of
Definition \ref{integral}; given three reduced cycles, $A_1$, $A_2$ and $A_3$, we say
$A_3=A_1\cup A_2$ if for any \'etale open $U\to X$, $U\times_{\sF} A_3=U\times_{\sF} A_1\cup U\times_{\sF} A_2$
as cycles in $\sF|_U$, as defined in Definition \ref{integral};
we call $A$ integral if it is not the union of two distinct non-trivial reduced cycles.
We define the cycle group $Z\lsta \sF$ be (rational) linear combinations of integral cycles of $\sF$. 
\end{defi}

Since $\sF$ is a sheaf over $X$, by descent, $A\sub \sF$ a reduced cycle if it is given by an
\'etale covering $U\lalp\to X$ and
reduced cycles $A\lalp\sub \sF|_{U\lalp}$
such that over each $U\lalpbe$:
\beq\label{re0} A\lalp\times_{U\lalp}U\lalpbe = A\lbe\times_{U\lbe}U\lalpbe \sub \sF|_{U\lalpbe}.
\eeq

\vsp
We continue to work with the semi-perfect obstruction theory $\phi$ as in Definition \ref{g-semi}.
We construct a group homomorphism
\beq\label{jmath}
\eta\lsta: Z\lsta N_{X/\cM} \lra Z\lsta \Ob_\phi
\eeq
by patching the collection 
\beq\label{j-alp}
\eta_{\phi\lalp\ast}:  Z\lsta N_{U\lalp/\cM} 
\lra  Z\lsta \Ob_{\phi\lalp}.
\eeq


\begin{lemm} Given an integral Artin substack $[A]\sub Z\lsta N_{X/\cM}$,
the collection 
$$[A\lalp]\defeq \eta_{\phi\lalp\ast}[A\times_X U\lalp]\in Z\lsta \Ob_\phi|_{U\lalp}
$$ 
satisfies the descent condition \eqref{re0} to form an integral
cycle in $Z\lsta \Ob_{\phi}$. 
\end{lemm}

\begin{proof}
This follows from the definition of semi-perfect obstruction theories and Proposition \ref{glue-cycle}.
\end{proof}

We denote the resulting cycle by $\eta_{\phi\ast}[A]$. We denote by
$$\eta_{\phi\ast}: Z\lsta N_{X/\cM}\lra Z\lsta \Ob_{\phi}
$$
the homomorphism by linear extension.
Applying this to the cycle $[\cC_{X/\cM}]\in Z\lsta N_{X/\cM}$, we define
\beq\label{vir-cycle}
[\fc_{X/\cM}]=\eta_{\phi\ast}[\cC_{X/\cM}]\in Z\lsta \Ob_\phi.
\eeq
\vsp

Let $s$ be the zero section of $\Ob_\phi$. To define the virtual cycle of $X$, we need to construct
a Gysin map $s^!: Z\lsta \Ob_\phi\lra  A\lsta X$.
The Gysin map for a bundle stack is constructed in \cite{Kre}.
The construction given here was first introduced in the work of the second named author in \cite{Li2}; it is reiterated
in \cite{KL}. We now sketch its construction.

Assume $X$ is proper and of finite type. Let $\sF$ be a coherent sheaf of $\sO_X$ modules, considered as a
sheaf stack.

Given a non-trivial integral cycle $A\sub \sF$, where $\sF$ is a sheaf stack of a coherent sheaf $\sF$ of $\sO_X$-modules,
we pick an affine scheme $U$ and an \'etale $\rho: U\to X$ so that $A\times_X U\ne \emptyset$. Since $U$ is affine, we can find
a vector bundle $\sV_U$ on $U$ and a surjective $\sV_U\to \sF|_U$. By definition, 
$$A|_U \times_{\sF|_U}\sV_U\sub \sV_U
$$
is a reduced Zariski closed subset. We let $B\sub U$ be the image of $A|_U \times_{\sF|_U}\sV_U$
under the projection $\sV_U\to U$. We let $Y\sub X$ be the closure of the image $\rho(B)$. 
Because $A$ is integral, $Y$ is integral; because $X$ is proper, $Y$ is proper. 

Because $X$ is a Deligne-Mumford stack, we can find a projective variety $S$ and a generic finite morphism
$f: S\to Y$. Then we can find a locally free sheaf $\sV$ on $S$ and a surjective sheaf homomorphism
$\sV\to f\sta \sF$. Pick an open $S_0\sub S$ so that $f|_{S_0}: S_0\to Y$ is \'etale. Let $D_0=\sV|_{S_0}\times_{\sF}A\sub \sV|_{S_0}$, 
and let $D_A\sub \sV$ be the closure of $D_0$. We call 
\beq\label{prop-rep}
(f:S\to Y, \sV\to f\sta\sF, D_A\sub \sV)
\eeq
a proper-representative of the integral cycle $A$. 
Let $e$ be the degree of $f: S\to Y$.
We define
$$s^![A] =e\upmo f\lsta \bl 0_{\sV}^![D_A]\br\in A\lsta X,
$$
where $0_{\sV}^!: Z\lsta \sV\to A\lsta S$ is the Gysin map of the zero section of $\sV$.

\begin{prop}
The stated procedure defines a Gysin map
\beq\label{Gysin}
s^!: Z\lsta \sF\lra A\lsta X
\eeq
by linear extension.
\end{prop}

\begin{proof}
We need to show that the map $s^!([A])$ of an integral $A\in Z\lsta \sF$ is independent of the choice of the proper-representatives
of $A$. This is essentially proved in \cite{Li2} and \cite[Sect. 3]{KL}. We outline the main idea here. 
Let $(f',S', \sV', D_A')$ be another proper representative of $A$. Then we can find a third proper representative 
$(\bar f, \bar S, \bar \sV, \bar D_A)$ that fits into the
commutative squares
\beq\label{compare}
\begin{CD}\bar S @>{g'}>> S'\\
@VV{g}V @VV{f'}V\\
S @>{f}>> Y
\end{CD}
\qquad and\qquad 
\begin{CD}
\bar \sV @>{\rho'}>>g\psta \sV'\\
@VV{\rho}V @VV{}V\\
g\sta \sV @>{}>> g\sta f\sta \sF=g\psta f\psta\sF=\bar f\sta \sF
\end{CD}
\eeq
Since $\bar f$ is generically finite, both $g$ and $g'$ are generically finite.
Then by the construction of the cycles $D_A$, $D'_A$ and $\bar D_A\sub \bar\sV$, we have
$\rho\lsta [\bar D_A]=d(g) [D_A]$ and $\rho'\lsta [\bar D_A]=d(g')[D_A']$. 
Therefore, for $s'$ and $\bar s$ the zero sections of $\sV'$ and
$\bar\sV'$, we have 
$$(\deg g)f\lsta s^{!}[D_A]=g\lsta \bar s^![\bar D_A]=(\deg g')f'\lsta s^{\prime !}[D_A'].
$$
Pushing forward to $A\lsta X$, we prove that \eqref{Gysin} is well-defined.
\end{proof}

We prove that the Gysin map preserves the rational equivalence.
Like the usual rational equivalence,
an integral rational equivalence is a pair $(A,h)$ of an integral cycle $A$ of $\sF$ and a non-trivial
rational function on $A$, which we define now.

\begin{defi}
We define a proper representative of a rational function on an integral $A\in Z\lsta \sF$ be a
proper representative $(f,\sV, D_A)$ of $A$ as in \eqref{prop-rep} and a rational function $h_f\in \kk(D_A)$
so that for $D_A^{\mathrm{nor}}$ the normalization of $D_A$, $\ti h_f$ the extension of $h_f$ to $
D_A^{\mathrm{nor}}$, and $\iota: D_A^{\mathrm{nor}}\to \sF$ 
the tautological morphism, we have that for any closed $p\in D_A^{\mathrm{nor}}$, the restriction
$\ti h_f|_{\iota\upmo(\iota(p))}$ is either nowhere defined or takes a single value.

Let $(f', S', \sV', D_A')$ and $h'_{f'}\in \kk(D_a')$ be another proper representative of a rational function
on $A$. We say $h_f\sim h_{f'}'$ if for a third proper representative $(\bar f, \bar S,\bar\sV, \bar D_A)$
of $A$ fitting into the commutative
squares \eqref{compare} so that 
$$(g'|_{\bar D_A})\sta(h'_{f'})=(g|_{\bar D_A})\sta(h_f)\in\kk(\bar D_A).
$$
We define a rational function on $A$ be an equivalence class of proper representatives of rational functions of $A$.
\end{defi}

We define an integral rational equivalence be a pair $(A,h)$ of an integral cycle $A$ of $\sF$ and a non-trivial 
rational function $h$ on $A$.
We define $W\lsta \sF$ be the (rational) linear combinations of integral
rational equivalences of $\sF$. 
We now define the boundary homomorphism
\beq\label{Par}
\partial: W\lsta \sF\to Z\lsta \sF.
\eeq

Let $(A,h)$ be a rational equivalence with $A$ integral and with a proper presentation
$(f,S, \sV, D_A)$ as in \eqref{prop-rep} and $h_f\in \kk(D_A)\sta$. We now construct $\partial(A,h)$. 
We first express 
\beq\label{par-}
\partial(D_A, h_f)=\sum_{A\in I} n_a [D_a],
\eeq
where $D_a\sub \sV$ are integral.  Let $\zeta_S: \sV\to f\sta \sF$ 
be the tautological map and let $\zeta_X: f\sta \sF\to \sF$
be the projection. By our definition of $h_f$, for each $D_a$, the $B_{a,S}=\zeta_S(D_a)\sub f\sta\sF$ has
the property that $D_a=\zeta_S\upmo(B_{a,S})$. Hence following the proof of Proposition \ref{glue-cycle}, one checks that 
$B_{a,S}$ is an integral cycle. 
We let $B_a=\zeta_X(B_{a,S})\sub \sF$; it is an integral cycle in $\sF$. 
Let $e=\deg f$; after defining $e_a$, which should be the degree of 
$\zeta_X|_{B_{a,S}}: B_{a,S}\to B_a$, we define
\beq\label{bound}
\partial(A,h)=e\upmo \sum_{a\in I} n_ae_a [B_a].
\eeq

{We define the degree $e_a$ using \'etale representative of $(A,h)$. 
We let $\pi: \sF\to X$ be the projection. Since $B_{a}\sub \sF$ is an integral cycle,
$\pi(B_a)\sub Y\sub X$ is an integral substack. We pick an affine $S'$ and an 
\'etale $f': S'\to Y$ so that $S'\times_Y \pi(B_a)\to \pi(B_a)$ is dominant. We pick a locally free sheaf
$\sV'$ on $S'$ and a surjective $\sV'\to f^{\prime\ast} \sF$. By shinking $S'$ if necessary, we can assume
$\rank \sV=\rank\ \sV'$. We let $D_A'=A\times_{\sF}\sV'\sub \sV'$.

To compare $D_A$ and $D_A'$, we form $T=S\times_Y S'$; we let $p: T\to S$ and $p': T\to S'$ be
the projections. We then pick an affine opens $U\sub T$ so that if we let $\varphi: \sV\to S$ be the projection,
$U\cap (\varphi(D_a)\times_Y S')$ is dense in $\varphi(D_a)\times_Y S'$; we pick a dense affine open $U_0\sub U$ 
so that $p'|_{U_0}: U_0\to S'$ is \'etale.
Since $U$ is affine, we can find an isomorphism
$p\sta\sV|_U\cong p\psta\sV'|_U$ that commutes with the projection $p\sta \sV\to p\sta f\sta\sF$ and
$p\psta\sV'\to p\psta f\psta \sF$:
$$ \begin{CD}
U_0\sub U\sub T @>{p'}>> S'\\
@VV{p}V@VVV\\
S @>>> Y
\end{CD}
\ \and\ 
\begin{CD}
p\psta\sV'|_U @>>> p\psta f\psta\sF\\
@VV{\cong}V @|\\
\ti\sV\defeq p\sta \sV|_U @>>> p\sta f\sta\sF
\end{CD}
$$

We now view
$D_A\times_S U$ and $D_A'\times_{S'}U$
as subsets in
$\ti\sV$ using the isomorphisms above, and using that $U_0\to Y$ is \'etale. Since $U_0\to Y$ is \'etale, we have that
$$D_A\times_SU_0=  D_A'\times_{S'}U_0= A\times_{\sF}\ti \sV|_{U_0}  \sub \ti \sV|_{U_0}.
$$
We let $\ti D_A$ be the closure of $D_A\times_SU_0$ in $\ti \sV$. 
We let $q: \ti \sV\to \sV$ and $q': \ti\sV\to \sV'$ be the projections.

By our assumption on $h_f\in \kk(D_A)\sta$, we see that $(q|_{\ti D_A})\sta(h_f)$
descends to a rational function on $D_A'$; we denote the descent by $h_{f'}\in \kk(D_A')\sta$. 
We express
$$\partial(D_A', h_{f'})=\sum_{b\in I'} n_b'[D_b'],
$$
where $D_b'$ are integral.

Finally, we pick a $b(a)\in I'$ so that the intersection 
\beq\label{inters}
\ti D_a\defeq (D_a\times_SU)\cap (D_{b(a)}'\times_{S'}U)\sub \ti\sV,
\eeq
where we view both $D_a\times_SU$ and $D_{b(a)}'\times_{S'}U$ as subsets of $\ti\sV$ using
the isomorphisms in the square above, $\ti D_a$ dominates both $D_a$ and $D_{b(a)}'$.
Because $h_{f'}$ is the descent of $(q|_{\ti D_A})\sta(h_f)$, a direct checking shows that
such $b(a)\in I'$ exists. 

The geometric meaning of this construction mimics the fiber product over $\sF$. We let
$\bar B_a\sub \sF|_S$ be the image stack of $D_a$ under $\sV\to f\sta\sF=\sF|_S$. 
We let $B_{b(a)}'\sub \sF|_{S'}$ be the image of $D_{b(a)}'$ under $\sV'\to f\psta\sF=\sF|_{S'}$. 
Our assumption that \eqref{inters} dominates $D_a$ and $D_{b(a)}'$ ensures that
the image of $B_{b(a)}'$ under $\sF|_{S'}\to \sF$ is $B_a$. (Note that $\bar B_a$ maps to $B_a$ because $B_a$ is
the image of $D_a$.) 
This way, the degree of $\bar B_a\to B_a$ 
is the same as the degree of 
\beq\label{def-F}
\bar B_a\times_{\sF} B_{b(a)}'\lra B_{b(a)}'.
\eeq
To define this degree, we pull back the first term in \eqref{def-F} to the bundle $\ti \sV$ over $U$ to 
obtain the $\ti D_a$ in \eqref{inters}; we pull back the second term in \eqref{def-F} to $\sV'$ to obtain $D_{b(a)}'\sub \sV'$.
Thus the degree of \eqref{def-F} is the same as
$$e_a\defeq \deg( q'|_{\ti D_a}: \ti D_a\lra D_{b(a)}').
$$

Since this definition uses the fact that the degree of a map is preserved after an \'etale base change of both
the domain and the target, it implies that $e_a$ is well-defined, independent of the choice of the
\'etale cover $S'\to Y$ we pick. Since the checking is routine, we omit it here.
}

With $e_a$ defined, we define $\partial(A,h)$ using \eqref{bound}.
By linear extension, we obtain the boundary operation \eqref{Par}.

\begin{coro} We have the relation
$$s^!\circ\partial=0: W\lsta \sF\lra A\lsta X.
$$
\end{coro}

\begin{proof}
We only need to check that for any integral $(A,h)\in W\lsta \sF$, we have
$(s^!\circ\partial) (A,h)=0$.
The proof is routine using proper representative of $(A,h)$, which transform this
identity to the the identity $s^!\circ\partial$ for rational equivalence in a vector bundle over a scheme.
Since the proof follows the argument in \cite{Li2} and \cite[Sect. 3]{KL}, we will omit the details here.
\end{proof}

\begin{defi-theo}\label{vir} 
Let  $\cM$ be an Artin stack locally of finite type, and let $X$ be a proper Deligne-Mumford stack 
of finite presentation. Suppose $\cM$ is smooth and of pure dimension, and 
suppose $\phi$ is
a semi-perfect relative obstruction theory of $X/\cM$.
The stated procedure produces 
a (virtual normal) cone cycle $[\fc_{X/\cM}]\in Z\lsta \Ob_\phi$. 
We define the virtual cycle of $X$ be
$$[X,\phi]\virt:=s^! [\fc_{X/\cM}]\in A_{\ast}X.
$$
\end{defi-theo}

We prove that the virtual cycle $[X,\phi_\bullet]\virt$ is deformation invariant in the sense of cycles.    
Consider a fiber-diagram (cf. \cite[Sect. 7]{BF}) of separated Deligne-Mumford stacks $X$ and $X'$ to smooth Artin stacks
$\cM$ and $\cM'$
$$
\begin{CD}
X'@>{u}>> X\\
@VVV @VVV\\
\cM'@>{v}>> \cM
\end{CD}
$$
such that $\cM$ and $\cM'$ are of pure dimensions, $u$ is representable, and $v:\cM'\to \cM$ is a regular immersion.
Let $v^!: A\lsta X\to A\lsta X'$ be the Gysin homomorphism associated to this square.

\begin{prop}\label{fiber}
A semi-perfect relative obstruction theory $\phi{}$ of $X/\cM$
induces a semi-perfect relative obstruction theory $\phi'$ of 
$X'/\cM'$, and their virtual cycles are related by
$$v^![X,\phi]\virt=[X',\phi{}']\virt.
$$
\end{prop}

\begin{proof} We let $\phi{}$ be given by $\{\phi\lalp,U\lalp, E\lalp, \psi\lalpbe\}_\Lam$.
We cover $X'$ by $U\lalp'=U\lalp\times_X X'$. 
Let $u\lalp$ be the induced morphism fitting into the Cartesian square
\beq\label{U-sq}
\begin{CD}
U\lalp' @>{u\lalp}>> U\lalp\\
@VVV @VVV\\
\cM' @>{v}>> \cM.\\
\end{CD}
\eeq
For each $\alpha\in\Lam$, we let
$$\phi'\lalp=u\lalp\sta\phi\lalp :E'\lalp:=u\lalp^\ast E\lalp \mapright{}
u\lalp^\ast L_{U\lalp/\cM}\lra
L_{U'\lalp/\cM'},
$$
where the last arrow is induced from the Cartesian square \eqref{U-sq}.

According to \cite[Prop. 7.2]{BF}, $\phi\lalp'$ is a perfect obstruction theory of $U'\lalp\to\cM'$.
To form a semi-perfect obstruction theory of $X'\to\cM'$, we need
transitions $\psi\lalpbe'$. 
Over $U\lalpbe'=U\lalp'\times_{X'}U'\lbe=X'\times_X U\lalpbe$, let 
$u_{\alpha\beta}:U'\lalpbe\to U\lalpbe$
be the tautological inclusion.
We define
$$\psi\lalpbe':=(u_{\alpha\beta})^\ast (\psi\lalpbe): E'\lalp|_{U'\lalpbe}\mapright{} E'\lbe|_{U'\lalpbe}. 
$$
Since $\phi{}$ is a semi-perfect obstruction theory, one checks that
$\phi{}'=\{U\lalp',\phi'\lalp,E\lalp',\psi'\lalpbe\}$
is a semi-perfect obstruction theory of $X'/\cM'$. Since the checking is routine, we will omit the
detail.

We now prove the identity $v^![X,\phi]\virt=[X',\phi']\virt$. 
%
Let $L=N_{\cM'/\cM}$ be the normal bundle to $\cM'$ in $\cM$. Because $v:\cM'\to\cM$ is a regular closed immersion,
$V$ is a vector bundle.
By the proof of \cite[Prop. 7.2]{BF} and \cite{Kresch2}, the Vistoli's rational equivalence 
$$(R,h) \in W\lsta(L\times_\cM \cC_{X/\cM}|_{X'})\sub W\lsta( L\times_{\cM'}N_{X/\cM}|_{X'})
$$
gives the following identity as cycles:
\beq\label{rat}
\partial (R,h)=[\cC_{\cC_{X/\cM}|_{X'}/\cC_{X/\cM}}]-[L\times_{\cM'} \cC_{X'/\cM'}]\in 
Z\lsta (L\times_{\cM'}N_{X/\cM}|_{X'}).
\eeq
In case $X/\cM$ has a perfect relative obstruction theory, push-forward this relation 
to the bundle-stack of the obstruction complex proves the desired
identity $v^![X,\phi]\virt=[X',\phi']\virt$.

We now show that \eqref{rat} proves the same identity in the case of semi-perfect obstruction theory.
We let 
$$\sF\defeq L\times_{{\sO_{\cM'}}}\Ob_{X/\cM}|_{X'},
$$
which is $L\times_{\sO_{\cM'}}\Ob_{X'/\cM'}$
since $\Ob_{X'/\cM'}=\Ob_{X/\cM}|_{X'}$. 

We let 
$$\eta\lalp: L\times_{\cM'} N_{X/\cM}|_{U\lalp'}\lra \sF|_{U\lalp'}= L\times_{\cM'}\Ob_{U\lalp'/\cM'}
$$
be the morphism (of stacks) induced by
$(\LL^{\ge -1}_{U\lalp/\cM})\dual\to E\lalp\dual\to H^1(E\lalp\dual)$. 
Repeating the proof of Proposition \ref{glue-cycle}, we conclude that the collection of image cycles $\eta_{\alpha\ast}[L\times_{\cM'} \cC_{X'/\cM'}]$ 
(resp. $\eta_{\alpha\ast}[\cC_{\cC_{X/\cM}|_{X'}/\cC_{X/\cM}}|_{U\lalp'}]$)
forms a cycle in $\sF$. Obviously, the former form the cycle $L\times_{\cM'}\fc_{X'/\cM'}$;
we denote the later by $[\mathfrak b]\in Z\lsta \sF$. Letting $s_0$ be the zero section of the stack $\Ob_{X'/\cM'}$, we conclude
$$s^! [L\times_{\cM'} \fc_{X'/\cM'}]=s_0^![\fc_{X'/\cM'}]= [X',\phi']\virt\in A\lsta X'.
$$
Mimic the proof given in \cite{Li2} and \cite[Sect. 3]{KL}, we conclude that
$$s^![\mathfrak b]=v^![X,\phi]\virt\in A\lsta X'.
$$

Finally, like before we can push $R$ via $\eta_{\alpha\ast}$ to form a cycle in $\sF$; 
we then check that $h$ descends to a rational function on this cycle,
resulting a rational equivalence $\beta\in W\lsta \sF$.
Then the relation \eqref{rat} gives 
$$\partial \beta=[\mathfrak b]-[L\times_{\cM'}\fc_{X'/\cM'}]\in Z\lsta \sF.
$$
Since the argument is routine, we will omit the details here.

Combined, we have
$$v^![X,\phi]\virt=s^![L\times_{\cM'}\fc_{X'/\cM'}]=s^![\mathfrak b]-s^!(\partial\beta)=[X',\phi']\virt.
$$
This proves the Theorem.
\end{proof}

    

\section{virtual cycle of derived objects and deformation invariance}
\def\Spl{\mathrm{{Splcpx}}}
\def\labc{_{\alpha\beta\gamma}}

In this section, we construct semi-perfect obstruction theory of the moduli of derived objects
on a projective Calabi-Yau threefold. In case the moduli space has an open, proper Deligne-Mumford substack,
the virtual class of the semi-perfect obstruction defines the Donaldson-Thomas invariant of this moduli space.
This for instance apply to the moduli spaces constructed in  \cite{Jason}.

We fix a smooth family of projective Calabi-Yau threefolds $S\to B$. 
We follow the convention that for $T\to B$, we use
$p_S$ and $p_T$ to denote the projections of $S\times_B T$ to $S$ and $T$. 

We quote a theorem of Inaba \cite{Inaba}, generalized by Lieblich \cite{Lieb}.

\begin{theo}\label{thm3.1}
Let $\fD_{S/B}$ be the stack of objects $E$ in $D^b(S)$ that are relatively perfect over $B$ and that 
for all geometric points $b\in B$ we have $\Ext^{i<0}(E_b,E_b)=0$ and $\Ext^0(E_b,E_b)=\kk(b)$. Then $\fD_{S/B}$ is a 
Deligne-Mumford stack locally of finite presentation over $B$.
\end{theo}

Recall that since $S\to B$ is smooth and projective, an $E\in D^b(S\times_B T)$ is
relatively perfect if it is locally perfect \cite{Lieb}.
In this section, we are interested in the substack of derived objects with fixed determinant line bundle. 

\begin{defi}
Let $E\in D^b(S\times_B T)$ be relatively perfect, and let $L$ be a line bundle on $S$. 
We say $\det \cE\sim p_S\sta L$ if there is a line bundle $J$
on $T$ such that
$\det E\cong p_S\sta L\otimes p_{T}\sta J$.
\end{defi}

Following \cite{Inaba}, we introduce the moduli functor
$$\Spl^L_{S/B} : (\text{Sch}/B)\lra (\text{Sets})
$$
that sends any $B$-scheme $T$ to
the set of all $E\in D^b(S\times_B T)$ satisfying the requirements in 
Theorem \ref{thm3.1} and $\det E\sim p_S\sta L$.
Applying the proof in \cite{Inaba} (see also \cite{Lieb}), its \'etale sheafification $\fD_{S/B}^L$ is a
Deligne-Mumford stack, locally of finite presentation over $B$.

We introduce the notion of semi-families and show that $\fD_{S/B}^L$ admits a 
universal semi-family.

\begin{defi}\label{semi-family}
For any Deligne-Mumford stack $X\to B$, a semi-family of derived objects in $D^b(S)$ on $X$ 
consists of an \'etale open covering
$\{U\lalp\}_\Lambda$ of $X$, derived objects $E\lalp\in D^b(S\times_B U\lalp)$, and 
quasi-isomorphisms
$$f\lalpbe: E\lalp|_{S\times_B U\lalpbe} \cong E\lbe|_{S\times_B U\lalpbe}\quad \mathrm{in}\quad D^b(S\times_B U\lalpbe)
$$
that satisfy the semi-cocycle condition: for any triple $(\alpha,\beta,\gamma)$ in $\Lambda$, there is a $c_{\alpha\beta\gamma}\in
\Gamma(\cO_{U_{\alpha\beta\gamma}}\sta)$ such that
\beq\label{co-3}
\pi_{\gamma\alpha}\sta f_{\gamma\alpha}\circ \pi_{\beta\gamma}\sta f_{\beta\gamma}\circ \pi\lalpbe\sta f\lalpbe
=c_{\alpha\beta\gamma}\cdot \id: E\lalp|_{S\times_B U_{\alpha\beta\gamma}}  \lra  E\lalp|_{S\times_B U_{\alpha\beta\gamma}} , 
\eeq
where $\pi\lalpbe: S\times_B U_{\alpha\beta\gamma}\to S\times_B U\lalpbe$ is the tautological projection.
\end{defi}

\begin{defi}
Let $X\sub \fD_{S/B}^L$ be an open substack.  A universal semi-family over $X$
consists of a semi-family $\{E\lalp, f\lalpbe\}_\Lam$ on $X$
of which the following holds: for any scheme $T$ over $B$ and any object
$F\in \Spl_{S/B}^L(T)$ such that the induced morphism $T\to \fD_{S/B}^L$ (induced by $F$) factors through $X\sub 
\fD_{S/B}^L$, then there is an \'etale cover $\{T\lalp\}_\Lam$ of $T$ (indexed by the same $\Lam$ in
$\{U\lalp\}_\Lam$) and morphisms $\varphi\lalp: T\lalp\to U\lalp$ such that
$$(1_S\times \varphi\lalp)\sta E\lalp\cong F|_{S\times_B U\lalp}.
$$
\end{defi}

\begin{prop}
Let $X\sub \fD_{S/B}^L$ be an open and closed substack. Then $X$ admits a universal semi-family.
\end{prop}

\begin{proof}
The proof follows from that $\fD_{S/B}^L$ is a sheafification of $\Spl_{S/B}^L$.
\end{proof}

Using universal semi-families of $X$, and applying the Atiyah class constructed by
Huybrecht-Thomas \cite{HT}, we construct semi-perfect obstruction theory of $X/B$.

\begin{lemm}\label{lem3.6}
Let $X\sub \fD_{S/B}^L$ be an open substack. Then $X\to B$ has a 
semi-perfect relative obstruction theory given by the Atiyah class constructed in \cite{HT}.
\end{lemm}

\begin{proof}
Let $\{E\lalp,f\lalpbe\}_\Lam$ be a universal semi-family. Using that it is universal locally, we conclude that for any closed 
$x\in U\lalp$, the family $E\lalp$ restricted
to the formal completion of $U\lalp$ at $x$ is the universal family of the hull of the infinitesimal deformations of $x$ in $U\lalp$. Thus applying
the work of Huybrecht-Thomas \cite{HT}, after fixing a closed embedding $U\lalp\to W\lalp$ into a $B$-smooth scheme $W\lalp$
and presenting the truncated cotangent complex of $U\lalp\to B$ as
$$L_{U\lalp/B}=[I_{U\lalp\sub W\lalp}/I^2_{U\lalp\sub W\lalp}\to \Omega_{W\lalp/B}|_{U\lalp}],
$$
we obtain a perfect relative obstruction theory
$$\phi\lalp: F\lalp\defeq \text{RHom}_{\cO_{U\lalp}}(E\lalp, E\lalp)_0\lra L_{U\lalp/B},
$$
where the subscript $0$ stands for the traceless part. 

Using $f\lalpbe$, we obtain an isomorphism
$g\lalpbe$ as shown making the square 
$$
\begin{CD}
F\lalp|_{U\lalpbe} @>{\phi\lalp|_{U\lalpbe}}>> L_{U\lalpbe/B}= L_{U\lalp/B}|_{U\lalpbe} \\
@VV{g\lalpbe}V@|\\
F\lbe|_{U\lalpbe}  @>{\phi\lbe|_{U\lalpbe}}>> L_{U\lalpbe/B}= L_{U\lbe/B}|_{U\lalpbe} 
\end{CD}
$$
commutative in $D(U\lalpbe)$.
Let 
$$\psi\lalpbe\defeq H^1(g_{\beta\alpha}\dual): H^1(F\lalp\dual)|_{U\lalpbe}\lra H^1(F\lbe\dual)|_{U\lalpbe}.
$$
Because a scaling automorphism $c\lalp\cdot \id: E\lalp\to E\lalp$, $c\lalp\in \Gamma(\cO_{U\lalp}\sta)$,
induces the identity automorphism of 
$\text{RHom}_{\cO_{U\lalp}}\!(E\lalp, E\lalp)_0$, 
the cocycle condition \eqref{co-3} implies that for 
any triple indices $(\alpha,\beta,\gamma)$, pullback to $U_{\alpha\beta\gamma}$ we obtain
identity
$$\psi_{\gamma\alpha}|_{U\labc}\circ  \psi_{\beta\gamma}|_{U\labc}\circ 
\psi_{\alpha\beta}|_{U\labc}=\id.
$$
This proves that $\{U\lalp,\phi\lalp,E\lalp,\psi\lalpbe\}_\Lam$ is a semi-perfect relative obstruction theory of $X\to\cM$.
\end{proof}

\begin{defi-theo}
In case $B$ is a closed point, $S$ is a smooth projective Calabi-Yau threefold and $L$ is a line bundle on $S$,
for any proper, open and closed
substack $X\sub \fD_S^L$, we define its associated Donaldson-Thomas invariant be $\deg [X]\virt$,
where $[X]\virt$ is the virtual cycle constructed by applying Definition-Theorem \ref{vir} to the semi-perfect obstruction theory of
$X$ constructed in Lemma \ref{lem3.6}.
\end{defi-theo}

Back to the case of a smooth family of projective Calabi-Yau $S\to B$ over a smooth base $B$,
we let $X\sub \fD_{S/B}^L$ be a $B$-proper, open and closed substack.
For any closed point $b\in B$, we denote $X_b=X\times_B b$, denote $i_b: X_b\to X$ the inclusion and $i_b^!:
A\lsta X\to A\lsta X_b$ the
Gysin map.

\begin{coro} Let the situation be as stated, then 
$i_b^! [X]\virt=[X_b]\virt$.
\end{coro}

\begin{proof}
This follows from Proposition \ref{fiber}.
\end{proof}

This confirms the deformation invariance of Donaldson-Thomas invariant of proper moduli of
derived objects.

\vsp\section{Further comments}

One can define the same Donaldson-Thomas invariant of a 
moduli of derived objects
using Behrend's weighted Euler number \cite{Beh}, based on Huybrecht-Thomas' construction of perfect obstruction
theory \cite{HT}. However,
semi-perfect obstruction theory proves that the Donaldson-Thomas invariant of derived objects is deformation invariant. 
\vsp

We can weaken the assumption on semi-perfect obstruction theory by replacing (1) in
Definition \ref{g-semi} by that restricting $\Ob_{\phi\lalp}$ to the reduced part $(U\lalp)\lred$
of $U\lalp$ descend to a sheaf on the reduced part of $X$; the item (2) is unchanged since $\nu$-equivalent only requires the
restriction of the obstruction sheaves to the reduced part of the Deligne-Mumford stack. The results of
this paper hold true in this weaker version of semi-perfect obstruction theory. Since we do not see
immediate application of this, for notational simplicity, in the end we phrase the semi-perfect obstruction theory
relying on the full obstruction sheaf of the stack.

\vsp
In many applications,  the $\psi_{\alpha\beta}$ in  (\ref{g-semi}) are induced from
quasi-isomorphisms
$$\Psi\lalpbe: E\lalp|_{U_{\alpha\beta}}\mapright{} E_\beta|_{U\lalpbe}.
$$
If we further assume that these quasi-isomorphisms form a descent data for the stacks
$h^1/h^0(E\lalp\dual)$, namely there is a two-term perfect complex $E$ on $X$ such that
$E\lalp\cong E|_{U\lalp}$ and the quasi-isomprhism $\Psi\lalpbe$ is isomorphic to that induced by the identity
map of $E$, then we can use intersection theory on bundle stacks to define the virtual cycles.
Indeed, by assumption, $\Ob_\phi\cong H^1(E\dual)$, which induces the coarse moduli functor 
$h^1/h^0(E\dual) \to\Ob_{\phi}$. 
For the cycle $[\fc_{X/\cM}]\in Z\lsta \Ob_\phi$, by expressing it as rational combination of integral cycles 
$[\fc_{X/\cM}]=\sum_i n_i [\fc_i]$, we define $\ti\fc_i=\fc_i\times_{\Ob_\phi}h^1/h^0(E\dual)$, where
each $\ti\fc_i\in Z\lsta h^1/h^0(E\dual)$ is an integral cycle, and define
$$[\ti\fc_{X/\cM}]=\sum_i n_i[\ti\fc_i] \in Z\lsta h^1/h^0(E\dual).
$$
Let $\ti s$ be the zero section of $h^1/h^0(E\dual)$; using the Gysin map $\ti s^!$ defined in \cite{Kre},
we obtain
$$\ti s^![\ti\fc_{X/\cM}]=s^! [\fc_{X/\cM}]\in A_{\ast}X.
$$
This way, the intersection theory on Artin stacks can be applied directly.
\vsp

The last comment is on the relation of the perfect obstruction theory formulated 
by Tian and the second named author in \cite{LT} with the semi-perfect obstruction theory defined in this paper.
It can be shown that a ``family Kuranishi model" over an affine chart $U\lalp\to X$ constructed in \cite{LT}
induces a perfect obstruction $\phi\lalp: E\lalp\to L_{U\lalp/\cM}$ as formulated in \cite{BF}.
The $\nu$-equivalence over intersection $U\lalpbe$ follows from the definition of the perfect obstruction
theory in \cite{LT}. It is in this sense we say that the semi-perfect obstruction theory is a mixture of
the two versions of perfect obstruction theories formulated in \cite{LT} and \cite{BF}.

\bibliographystyle{amsplain}

\end{document}